\documentclass{amsart}

\usepackage{mathrsfs}
\usepackage{graphicx}

\usepackage{amssymb}
\usepackage{amsmath}
\usepackage{amsthm}

\usepackage[all,cmtip]{xy}

\usepackage{tikz}
\usepackage{tikz-cd}
\usepackage{float}
\usepackage{caption}

\pgfdeclarelayer{nodelayer}
\pgfdeclarelayer{edgelayer}
\pgfsetlayers{edgelayer,nodelayer,main}

\tikzstyle{none}=[inner sep=0mm] %透明顶点，用于放标签a,b以及作为黑箭头两端的顶点
\tikzstyle{new style 0}=[fill=white, draw=black, shape=circle]%顶点
\tikzstyle{new edge style 0}=[->, draw=red, line width=0.7mm]%红实线箭头
\tikzstyle{new edge style 1}=[->, draw=blue, line width=0.7mm]%蓝实线箭头
\tikzstyle{new edge style 2}=[->, draw=red, thick, dashed, line width=0.7mm]%红虚线箭头
\tikzstyle{new edge style 3}=[draw=blue, thick, dashed, ->, line width=0.7mm]%蓝虚线箭头
\tikzstyle{new edge style 4}=[->, line width=0.3mm]%黑实线箭头

\newtheorem{thm}{Theorem}[section]
\newtheorem{cor}[thm]{Corollary}
\newtheorem{lem}[thm]{Lemma}
\newtheorem{prop}[thm]{Proposition}

\theoremstyle{definition}
\newtheorem{defn}[thm]{Definition}
\newtheorem{exam}[thm]{Example}

\newtheorem*{mtm}{Main Theorem}
\theoremstyle{remark}
\newtheorem{rem}[thm]{Remark}
\numberwithin{equation}{section}
\newcommand{\Z}{\mathbb Z}

\newcommand{\To}{\longrightarrow}

\newcommand{\var}[1]{\textbf{\textsc{#1}}}

\newcommand{\PP}{\mathbb{P}}

\def\a{\alpha}

\begin{document}

\title[Supersolvable closures are finitely generated]{Supersolvable closures of finitely generated subgroups of the free group}

\author{Lida Chen}
\address{Department of Mathematics, Soochow University, Suzhou 215006, CHINA}
\email{cllddd@126.com}

\author{Jianchun Wu}
\address{Department of Mathematics, Soochow University, Suzhou 215006, CHINA}
\email{wujianchun@suda.edu.cn}

\thanks{}


\keywords{supersolvable, free groups, variety. \\ $*$The second author is the corresponding author.}

\begin{abstract}
We prove the pro-supersolvable closure of a finitely generated subgroup of the free group is finitely generated. It extends similar results for pro-$p$ closures proved by Ribes-Zalesskii and pro-Nilpotent closures proved by Margolis-Sapir-Weil.
\end{abstract}

\maketitle

\section{Introduction}
The pro-$\var{V}$ topology on a group is defined by Hall in \cite{Hall49, Hall50}, where $\var{V}$ is a  \emph{variety}, namely a class of finite groups closed with respect to subgroups, homomorphic images and  direct products. Such a class is usually called a \emph{pseudovariety}, see \cite{MSW01, RZ94, Weil98}.

The closure $Cl_{\var{V}}(S)$ of a subset $S$ in a finitely generated free group $F$ endowed with the pro-$\var{V}$ topology consists of those elements not able to be distinguished  with $S$ by any homomorphism from $F$ to the groups in $\var{V}$. %Namely, $Cl_{\var{V}}(S)=\displaystyle\bigcap_{\phi: F\to V \atop V\in \var{V}} \phi^{-1}\phi(S).$
For instance, let $\var{V}$ be the variety $\var{Ab}$ of all finite abelian groups, the closure of the trivial subgroup in $F$ is the derived subgroup $[F,F]$.

Hall proved that every finitely generated subgroup is closed in $F$ when $\var{V}$ is the variety of all finite groups. This is not true for the pro-$\var{Ab}$ topology---the trivial subgroup is not closed and its closure $[F,F]$ is not even finitely generated. It is natural to raise a question that given a variety $\var{V}$ and a finitely generated subgroup $H$, how to determine the pro-$\var{V}$ closure of $H$.

When $\var{V}$ is the variety $\var{G}_p$ ($p$ is prime) of all finite $p$-groups, Ribes-Zalesskii \cite{RZ94} proved that the pro-$\var{G}_p$ closure of a finitely generated subgroup is finitely generated and gave an algorithm to compute the closure. Margolis-Sapir-Weil modified Ribes-Zalesskii's algorithm in \cite{MSW01} to compute the pro-$\var{G}_p$ closure in polynomial time. They also proved when $F$ is endowed with pro-nilpotent topology, namely when $\var{V}$ is the variety $\var{Nil}$ of all finite nilpotent groups, the pro-$\var{Nil}$ closure of a finitely generated subgroup is finitely generated and computable.

The main tool used in \cite{MSW01, RZ94} is the bijective correspondence between finitely generated subgroups and certain graphs acted by groups which are called reduced inverse automata in \cite{MSW01} and  reduced $A$-labeled graphs in this paper where $A$ is a basis of $F$. The idea behind this correspondence goes back to Serre \cite{Ser77} and Stallings \cite{Sta83}, that has turned out to be extremely useful since then. This correspondence allows one to define the notion ``overgroups'' of a finitely generated subgroup. An overgroup of a finitely  generated subgroup $H$ is always finitely generated, but whether a subgroup is an overgroup of $H$ depends on the chosen basis of $F$. When a basis is fixed, there are only finitely many overgroups of $H$. In \cite{MSW01}, Margolis-Sapir-Weil proved the pro-$\var{Nil}$ closure of $H$ is the intersection of all pro-$\var{G}_p$ closures of $H$ and each pro-$\var{G}_p$ closure is an overgroup of $H$ for a fixed basis $A$. So the intersection involving finitely many finitely generated subgroups is finitely generated.
We will describe this tool in Section \ref{sp3} and refer to \cite{KM02, MVW01} where relevant statements and proofs can be found. For instance, a simple proof of Hall's extension theorem  in \cite{Hall50}, which states every finitely generated subgroup is a free factor of a subgroup of finite index in the free group can be found in \cite{KM02} Theorem 8.1.

Motivated by Margolis-Sapir-Weil's work, we focus on the case that $\var{V}$ is the variety $\var{Su}$ of all finite supersolvable groups. It is a variety lies between $\var{Nil}$ and the variety $\var{Sol}$ of all finite solvable groups.
We prove
\begin{mtm}%\ref{mainres}
Let $H$ be a finitely generated subgroup of $F$ endowed with the pro-$\var{Su}$ topology, then the closure of $H$ is finitely generated.
\end{mtm}

The proof depends on the following observation.  A finite nilpotent group is a direct product of its Sylow subgroups and so can project onto a group in the subvariety $\var{G}_p$ of $\var{Nil}$ for each prime $p$ dividing its order. This allows the authors of \cite{MSW01} to prove the pro-$\var{Nil}$ closure is the intersection of all pro-$\var{G}_p$ closures. A similar result holds for a finite supersolvable group, that is Lemma \ref{fp} which states that a finite supersolvable group can project onto a group in the subvariety $\var{H}_p$ (defined in Section \ref{sp}) of $\var{Su}$ for each prime $p$ dividing its order. It is the key to prove the pro-$\var{Su}$ closure of $H$ is the intersection of all pro-$\var{H}_p$ closures. Then we turn to prove each pro-$\var{H}_p$ closure is an overgroup of $H$ for a fixed basis, that is Corollary \ref{hpclo}, and finally prove the main theorem by the fact that there are only finitely many overgroups of $H$.

The organization of this paper is as follows: The definitions of some varieties with their notations
are given in Section \ref{sp1}; some easy but frequently used results about the pro-$\var{V}$ topology on a free group are listed in Section \ref{sp2}; in Section \ref{sp3} and \ref{sp4}, we give a brief introduction about the tool mentioned above and some important lemmas and  propositions  which are used in this paper; in Section \ref{t1} we show that the pro-$\var{H}_p$ closure of $H$ is an overgroup of $H$ (Corollary \ref{hpclo}) and provide a simple method to decide whether the pro-$\var{H}_p$ closure of a finitely generated subgroup is $F$ (Theorem \ref{hpdense}); in Section \ref{t2}, after introducing some notions and propositions about finite supersolvable groups, we prove the main theorem of this paper.

\section{Preliminary}\label{sp}

\subsection{Varieties of finite groups}\label{sp1}
%\subsubsection{Definition}
\begin{defn}
A class of finite groups is called a \emph{variety} if it is closed with respect to subgroups, homomorphic images and direct products. The \emph{product} $\var{U}*\var{V}$ of two varieties $\var{U}$, $\var{V}$ consists of all groups $G$ having a normal subgroup $N\in \var{U}$ such that $G/N\in \var{V}$. For two varieties $\var{U}, \var{V}$, we say $\var{U}$ is a \emph{subvariety} of $\var{V}$ if $\var{U}\subset \var{V}$.
\end{defn}

%\begin{rem}
% Such classes are often referred to as \emph{pseudovarieties} or \emph{formations}, see \cite{MSW01} or [].
%\end{rem}

%\begin{defn}
%For two varieties $\var{U}, \var{V}$, we say $\var{U}$ is a \emph{subvariety} of $\var{V}$ if $\var{U}\subset \var{V}$.
%\end{defn}

%\begin{defn}
%The \emph{product} $\var{U}*\var{V}$ of two varieties $\var{U}$, $\var{V}$ consists of all groups $G$ having a normal subgroup $N\in \var{U}$ such that $G/N\in \var{V}$. For two varieties $\var{U}, \var{V}$, we say $\var{U}$ is a \emph{subvariety} of $\var{V}$ if $\var{U}\subset \var{V}$.
%\end{defn}

%\subsubsection{Some vatieties}

\begin{defn}
A group $G$ is called a \emph{supersolvable} group if there exists a normal series:
$$1=H_0\leq H_1\leq \cdots\leq H_n=G$$ where each $H_i\lhd G$ and  each $H_{i+1}/H_i$ is cyclic.
\end{defn}

%\begin{thm}
%For a variety $\var{H}$ of finite supersolvable groups, the following are equivalent:
%
%(1) $\var{H}$ is freely indexed;
%
%(2) $\var{H}$ is Hall;
%
%(3) $\var{H}=\var{G}_p*\var{Ab}_d$ for some prime $p$ and some positive integer $d$ dividing $p-1$.
%\end{thm}

%\begin{defn}[Hall$_p$ group]
%Let $p$ be a prime number, a finite group $G$ is called a Hall$_p$ group if there exists a normal subgroup $N\lhd G$ such that $N$ is a $p$-group and $G/N$ is in the variety $\var{Ab}_{p-1}$.
%\end{defn}

The class of all finite supersolvable groups is a variety denoted by $\var{Su}$. We use the following notations for the varieties mentioned in this paper.

\vspace{2mm}

\begin{tabular}{|c|c|}
  \hline
  Notation & The corresponding variety\\
  \hline
%  $\var{G}$ & the variety of all finite groups \\

  $\var{G}_p$ & the variety of all $p$-groups ($p$ is prime)\\

  $\var{Ab}$ & the variety of all abelian groups \\
  $\var{Ab}_d$ & the variety of all abelian groups of  exponent  dividing $d$ \\
  $\var{Nil}$ & the variety of all nilpotent groups \\

  $\var{Su}$ & the variety of all supersolvable groups \\
  $\var{Sol}$ & the variety of all solvable groups \\

  $\var{H}_p$ & $\var{G}_p*\var{Ab}_{p-1}$ ($p$ is prime)\\
  \hline
\end{tabular}

\vspace{2mm}

\begin{rem}
It is easy to check $\var{G}_p\subset \var{Nil}\subset \var{Su} \subset \var{Sol}$. $\var{H}_p$ is also a subvariety of $\var{Su}$.
\end{rem}

%The variety of all supersolvable groups: $\var{Su}$;
%
%The variety of all solvable groups: $\var{Sol}$;
%
%The variety of all nilpotent groups: $\var{Nil}$;
%
%The variety of all $p$-groups: $\var{G}_p$;
%
%The variety of all abelian groups of  exponent  dividing $d$: $\var{Ab}_d$;
%
%The variety of all Hall$_p$ groups : $\var{Hall}_p$;

\subsection{The pro-$\var{V}$ topology on a free group}\label{sp2}
In \cite{Hall50},  Hall introduced a topology for  free groups which is called the pro-$\var{V}$ topology nowadays.

%\begin{defn}[Hall]
%Let $\mathcal{K}$ be a family of subgroups of a group $G$ such that
%
%(1) $\bigcap_{K\in \mathcal{K}}K=1$;
%
%(2) If $K, K'\in \mathcal{K}$ then $K\cap K'\in \mathcal{K}$;
%
%(3) If $K\in \mathcal{K}$ then there is a normal subgroup $K'\in \mathcal{K}$ such that $K'\subseteqq K$.

%We can define a topology on $G$ by taking $\{xK, Kx | x\in G, K\in \mathcal{K}\}$ as a basis for open sets in $G$. This topology is called the subgroup topology on $G$ defined by $\mathcal{K}$.
%\end{defn}

Let $F$ be a free group and  $\var{V}$ be a variety. Denote the family of normal subgroups $K$ of $F$ such that $F/K\in \var{V}$ by $\mathcal{K}$. We can define a topology on $F$ by taking $\{Kx~|~K\in \mathcal{K}, ~x\in F\}$ as a basis for open sets in $F$. Note that $\mathcal{K}$ is a basis of neighborhoods of the element $1$. This topology is called the pro-$\var{V}$ topology on $F$ which can also be defined for an arbitrary group.

\begin{rem}
The pro-$\var{V}$ topology on $F$ is the coarsest topology which makes all homomorphisms from $F$ to elements (equipped with the discrete topology) of $\var{V}$ continuous.
\end{rem}

In \cite{MSW01}, the authors gave another point view for the pro-$\var{V}$ topology. We describe it below.

 For $x, y\in F$, we say that a finite group $V$ separates $x$ and $y$ if there exists a homomorphism $\phi_V: F\to V$ such that $\phi_V(x)\neq \phi_V(y)$. Let $$d(x,y)=\begin{cases}
          e^{-r(x,y)}, & \mbox{if }x,y \mbox{ can be separated by an element in }$\var{V}$ \\
          0, & \mbox{otherwise}
        \end{cases}$$
where $r(x,y)=\min\{|V|: V\in \var{V} ~\text{separtes}~ x ~\text{and}~ y\}$.

One can verified that $d$ is a pseudometric on $F$. The topology defined by $d$ is the same as the pro-$\var{V}$ topology.

\begin{defn}
Let $\var{V}$ be a variety. A subgroup $H$ of $F$ is said to be \emph{$\var{V}$-open} (resp. \emph{$\var{V}$-closed}) if it is open (resp. closed) in $F$ endowed with the pro-$\var{V}$ topology. The closure of a subset $S\subset F$ % for the pro-$\var{V}$ topology
 is called the $\var{V}$-closure of $S$ and denoted by $Cl_{\var{V}}(S)$. Moreover $S$  is called $\var{V}$-dense in $F$ if $Cl_{\var{V}}(S)=F$.
\end{defn}

%\begin{prop}
%Let $H$ be a subgroup of $F$. Then the following are equivalent.
%
%(1) $H$ is $\var{V}$-closed;
%
%(2) $H$ is the intersection of all $\var{V}$-open subgroups containing it;
%
%(3) $H=\displaystyle\bigcap_{\phi: F\to V\in \var{V}}\phi^{-1}(\phi(H))$.
%\end{prop}

\begin{prop}[\cite{MSW01}, Proposition 1.3]
Let $H$ be a subgroup of $F$, then$$Cl_{\var{V}}(H)=\bigcap_{K \in \mathcal{K}} K=\bigcap_{\phi\in \Phi_\var{V}} \phi^{-1}(\phi(H)).$$
where $\mathcal{K}$ is the set of all $\var{V}$-open subgroups that containing $H$, and $\Phi_\var{V}$ is the set of all homomorphisms from $F$ to elements of $\var{V}$.
\end{prop}

Note that for each  $\phi\in \Phi_\var{V}$, $\phi(F)$ lies in $\var{V}$, then $\phi$ induces a surjective homomorphism $\phi_s: F\to \phi(F)\in \var{V}$ such that $\phi_s^{-1}(\phi_s(H))=\phi^{-1}(\phi(F))$. We have
\begin{cor}\label{dense}
Let $H$ be a subgroup of $F$, then $H$ is $\var{V}$-dense in $F$ if and only if for each surjective $\phi_s\in \Phi_\var{V}$, $\phi_s^{-1}(\phi_s(H))=F$, namely $\phi_s(H)=\phi_s(F)$.
\end{cor}
%\begin{proof}
%Note thatFor , $$
%\end{proof}

\begin{lem}[\cite{MSW01}, Corollary 3.1]\label{subd}
Suppose $\var{U},\var{V}$ are two varieties such that $\var{U}\subset \var{V}$. Let $H$ be a finitely generated subgroup of $F$.

(1) If $H$ is $\var{V}$-dense, then $H$ is $\var{U}$-dense.

(2) The $\var{V}$-closure of $H$ is contained in the $\var{U}$-closure of $H$.

(3) If $H$ is $\var{U}$-closed, then $H$ is $\var{V}$-closed.
\end{lem}

\subsection{Representation of subgroups by labeled graphs}\label{sp3}

A free group $F$ can be identified with the fundamental group of the bouquet of circles $R$ while each subgroup of $F$ corresponds to an associated graph covering $R$. In this point of view, Stallings introduced a useful notion of a folding of graphs in \cite{Sta83} which is widely used in the research of the lattices of subgroups of free groups from then on. This method is also useful to solve algorithmic problems concerning subgroups of $F$. We briefly describe this technique below and refer to \cite{KM02} \cite{MVW01} for more details.

\subsubsection{Reduced $A$-labeled graph}
Denote the free group $F$ over a finite alphabet $A$ by $F(A)$.

\begin{defn}
An \emph{$A$-labeled graph} $\Gamma$ is a directed graph satisfying the following conditions:

(1) the underlying undirected graph of $\Gamma$ is connected;

(2) there is a prechosen vertex marked by 1 and we call it the base vertex of $\Gamma$;

(3) each edge of $\Gamma$ is labeled by a letter of $A$.
\end{defn}

For an $A$-labeled graph $\Gamma$, if there exist two different edges $e_1, e_2$ with the same origin (resp. with the same terminus) and the same label, then we can identify $e_1, e_2$ to produce a new $A$-labeled graph. This operation is called ``folding of graphs'', see \cite{KM02} for details.

\begin{defn} [See \cite{MVW01}]
An $A$-labeled graph $\Gamma$ is said to be \emph{reduced} if different edges with the same origin (resp. with the same terminus) have different labels, and if each vertex except the base vertex is adjacent to at least two different edges.
\end{defn}

By a path $p$ in an $A$-labeled graph $\Gamma$ we understand a path in the underlying undirected graph of $\Gamma$ being allowed to travel back along edges. The label of $p$ is the word (maybe not a reduced word) in $F(A)$ obtained by concatenating consecutively the labels of the edges crossed by $p$, concatenating $a^{-1}$ whenever an edge labeled by $a\in A$ is crossed backwards. The path $p$ is said to be reduced if the label of $p$ is a reduced word in $F(A)$.

% If $p$ does not cross twice the same edge consecutively, once in one direction and then in the other, then there is no subword like $aa^{-1}$ or $a^{-1}a$ in the label of $p$, and so it is a reduced word in $F(A)$. The path $p$ is said to be reduced in this case.

For a reduced $A$-labeled graph $\Gamma$, we can associate a subgroup $L(\Gamma)$ of $F(A)$ with it which is the set of words labeling reduced paths in $\Gamma$ from the base vertex 1 back to itself. Note that if $\Gamma$ is finite, then the associated subgroup $L(\Gamma)$ is finitely generated.

Conversely, suppose $H=\langle h_1,...,h_k\rangle$  is a finitely generated subgroup of $F(A)$ where each $h_i$ is a non-empty reduced words over the alphabet $A$, then $H$ can be represented by a finite reduced $A$-labeled graph through the following steps (see Figure 1 for an example).

\textbf{Step 1}. Suppose $h_1=a_{i_1}^{\epsilon_1}\cdots a_{i_m}^{\epsilon_m}, a_{i_1},...,a_{i_m}\in A, \epsilon_1,...,\epsilon_m=\pm 1$, we construct a $m$ subdivided circle labeled by $h_1$ based on a vertex marked by 1, in which each inverse letter $a_{i_j}^{\epsilon_j}$ where $\epsilon_j=-1$  gives rise to a $a_{i_j}$-labeled edge in the reverse direction on the circle.

\textbf{Step 2}. For each $h_i \in \{h_2,...,h_k\}$, we do the same thing like that in the first step and construct a $m_i$ divided circle in addition to the vertex marked by 1 where $m_i$ is the length of $h_i$. Note that the circle corresponding to $h_i$ has $m_i$ edges and $m_i-1$ vertices. We now get a finite $A$-labeled graph $\Gamma'$ consisting of $k$ circles based on the common vertex 1.

\textbf{Step 3}. We do the operation ``folding of graphs'' and make $\Gamma'$ to be reduced, namely we iteratively identify different edges with the same origin (resp. with the same terminus) and the same label. This process terminates since $\Gamma'$ is a finite graph and it does not matter in which order identification take place. The resulting reduced $A$-labeled graph is denoted by $\Gamma_A(H)$.

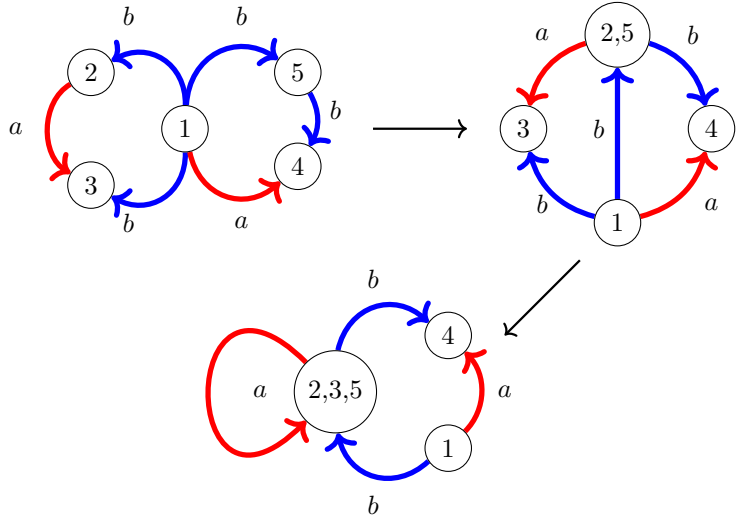
\begin{figure}[H]
	\centering
	\begin{tikzpicture}
		\begin{pgfonlayer}{nodelayer}
			\node [style=new style 0] (0) at (-10.5, 3.75) {1};
			\node [style=new style 0] (1) at (-11.75, 4.5) {2};
			\node [style=new style 0] (2) at (-11.75, 3) {3};
			\node [style=new style 0] (3) at (-9, 4.5) {5};
			\node [style=new style 0] (4) at (-9, 3.25) {4};
			\node [style=none] (5) at (-8, 3.75) {};
			\node [style=none] (6) at (-6.75, 3.75) {};
			\node [style=new style 0] (7) at (-4.75, 2.5) {1};
			\node [style=new style 0] (8) at (-4.75, 5) {2,5};
			\node [style=new style 0] (9) at (-6, 3.75) {3};
			\node [style=new style 0] (10) at (-3.5, 3.75) {4};
			\node [style=new style 0] (11) at (-8.5, 0.25) {2,3,5};
			\node [style=new style 0] (12) at (-7, 1) {4};
			\node [style=new style 0] (13) at (-7, -0.5) {1};
			\node [style=none] (14) at (-12.75, 3.75) {$a$};
			\node [style=none] (15) at (-11.25, 5.25) {$b$};
			\node [style=none] (16) at (-9.75, 2.5) {$a$};
			\node [style=none] (17) at (-9.75, 5.25) {$b$};
			\node [style=none] (19) at (-11.25, 2.5) {$b$};
			\node [style=none] (20) at (-6.25, 0.25) {$a$};
			\node [style=none] (21) at (-8.5, 4) {$b$};
			\node [style=none] (22) at (-8, -1.25) {$b$};
			\node [style=none] (23) at (-5.75, 2.75) {$b$};
			\node [style=none] (24) at (-5.75, 5) {$a$};
			\node [style=none] (25) at (-5, 3.75) {$b$};
			\node [style=none] (26) at (-3.5, 2.75) {$a$};
			\node [style=none] (27) at (-3.75, 5) {$b$};
			\node [style=none] (28) at (-9.5, 0.25) {$a$};
			\node [style=none] (29) at (-8, 1.75) {$b$};
			\node [style=none] (30) at (-5.25, 2) {};
			\node [style=none] (31) at (-6.25, 1) {};
		\end{pgfonlayer}
		\begin{pgfonlayer}{edgelayer}
			\draw [style=new edge style 1, bend right=60, looseness=1.25] (0) to (1);
			\draw [style=new edge style 1, bend left=60, looseness=1.25] (0) to (2);
			\draw [style=new edge style 0, bend right=60] (1) to (2);
			\draw [style=new edge style 0, bend right=60, looseness=1.25] (0) to (4);
			\draw [style=new edge style 1, bend left=60, looseness=1.25] (0) to (3);
			\draw [style=new edge style 1, bend left] (3) to (4);
			\draw [style=new edge style 4] (5.center) to (6.center);
			\draw [style=new edge style 1] (7) to (8);
			\draw [style=new edge style 1, bend left] (8) to (10);
			\draw [style=new edge style 1, bend left] (7) to (9);
			\draw [style=new edge style 0, bend right] (8) to (9);
			\draw [style=new edge style 0, bend right] (7) to (10);
			\draw [style=new edge style 0, in=-135, out=135, loop] (11) to ();
			\draw [style=new edge style 1, bend left=60, looseness=1.25] (11) to (12);
			\draw [style=new edge style 1, bend right=300, looseness=1.25] (13) to (11);
			\draw [style=new edge style 0, bend right=45] (13) to (12);
			\draw [style=new edge style 4] (30.center) to (31.center);
		\end{pgfonlayer}
	\end{tikzpicture}	
	
	\caption{Getting $\Gamma_A(H)$ for $H=\left \langle bab^{-1}, b^2 a^{-1} \right \rangle$ by ``folding of graphs''}\label{Stallings folding}
\end{figure}

%In \cite{KM02}, $\Gamma_A(H)$ is in fact the same as the core of coset graph of $H$ with respect to the base vertex, see \cite{KM02} Theorem 5.1.

A finite reduced $A$-labeled graph $\Gamma$ can be associated with a finitely generated subgroup $L(\Gamma)$ while a finite generated subgroup $H$ can be represented by a finite reduced $A$-labeled graph $\Gamma_A(H)$. In fact, one can prove $L(\Gamma_A(H))=H$ (see \cite{KM02} Theorem 5.1, 5.2) and we have the following proposition.

\begin{prop}
There is a one-to-one correspondence between finitely generated subgroups of $F(A)$ and finite reduced $A$-labeled graphs.
\end{prop}

\subsubsection{Scheier graphs and Scheier free factors}

There is another way to construct $\Gamma_A(H)$ described in \cite{KM02}. Let $\Gamma(F(A),A)$ be the Cayley graph of $F(A)$ with respect to the basis $A$ and $\Gamma_H= \Gamma(F(A),A)/H$ be the quotient graph of $\Gamma(F(A),A)$ by the right action of $H$ which is usually called Schreier graph or coset graph associated to $H$. The vertex set of $\Gamma_H$ is the set of cosets $\{Hg~|~g\in F(A)\}$, and there is an edge from  $Hg_1$ to $Hg_2$ labeled by $a\in A$ if and only if $Hg_1a=Hg_2$. The vertex corresponding to the coset $H$ is marked by 1, then $\Gamma_H$ is an $A$-labeled infinite graph unless $H$ is of finite index in $F(A)$. Now $\Gamma_A(H)$ is the core of $\Gamma_H$ with respect to the base vertex 1, see \cite{KM02} Theorem 5.1.

\begin{defn}
Let $F(A)$ be the free group with basis $A$ and the elements are understand as reduced words over the alphabet $A$. A prefix-closed subset of  $F(A)$ is called a \emph{Schreier set}. Let $H$ be a subgroup of $F(A)$, a complete Schreier set $T$ of representatives for all cosets of $H$ is called a \emph{Schreier transversal} for $H$ and its associated basis $B_T$ of $H$ is called a \emph{Schreier basis} of $H$. A \emph{Schreier free factor} (with respect to the chosen basis $A$) of $H$ is any subgroup of $H$ that is generated by a subset of some Schreier basis of $H$.
\end{defn}

We can obtain Schreier transversals and Schreier bases easily from the coset graph $\Gamma_H$ of $H$. Choose any spanning tree $L$ of $\Gamma_H$, then the set of all reduced words that label a path in $L$ starting at the base vertex 1 is a Schreier transversal $T$ for $H$. The Schreier basis $B_T$ of $H$ corresponding to $T$ can be obtained as follows. For each edge $e$ in $\Gamma_H$ that is not on the spanning tree $L$, consider the word $w_e:=w_{o_e} a_e w_{t_e}^{-1}$ where $o_e$  is the origin of $e$, $t_e$ is the terminus of $e$,  $a_e\in A$ is the label of $e$, and $w_{o_e}$ (resp $w_{t_e}$) is the representative in $T$ of the coset $o_e$ (resp. $t_e$). Then $B_T=\{w_e~|~e\in \Gamma_H-L\}$ is the corresponding Schreier basis of $H$.

\begin{exam}
Let $H=\left \langle aba^{-1}b,b^{-1}a^{-1}ba^{-1}b,a^{-1}b^{-1},b^{-1}ab^3 a^{-1}b \right \rangle$ be a finitely generated subgroup of $F(A)=\langle a, b\rangle$. $\Gamma_A(H)$ is shown in Figure \ref{Gamma(H) and L} which is also the core of the coset graph $\Gamma_H$ of $H$. A spanning tree $L_0$ of $\Gamma_A(H)$ is depicted by dotted lines which can be easily extended to a spanning tree $L$ of $\Gamma_H$. The corresponding Schreier transversal is
$$T=\{1,b^{-1}ab^{-1},b^{-1}a, b^{-1},b^{-1}ab, a^{-1},\cdots \},$$
and the Schreier basis of $H$ corresponding to $T$ is
			$$B_{T}=\{aba^{-1}b,b^{-1}ab^{-1}ab,b^{-1}abbba^{-1}b, ba \}.$$
\begin{figure}[H]
\centering
	\begin{tikzpicture}
	\begin{pgfonlayer}{nodelayer}
		\node [style=new style 0] (0) at (0, 0) {1};
		\node [style=new style 0] (1) at (2.5, 1.5) {2};
		\node [style=new style 0] (2) at (2.5, -1.25) {4};
		\node [style=new style 0] (3) at (5, 0) {3};
		\node [style=new style 0] (4) at (5, 2) {5};
		\node [style=new style 0] (5) at (-3, 0) {6};
		\node [style=none] (6) at (-1.75, 1.25) {$b$};
		\node [style=none] (7) at (-1.5, -1.25) {$a$};
		\node [style=none] (8) at (1, -1) {$b$};
		\node [style=none] (9) at (1, 1.25) {$a$};
		\node [style=none] (10) at (4, 1.25) {$b$};
		\node [style=none] (11) at (4, -1) {$a$};
		\node [style=none] (12) at (5.75, 1) {$b$};
		\node [style=none] (13) at (3.75, 2.5) {$b$};
		\node [style=none] (14) at (2.75, 0) {$a$};
	\end{pgfonlayer}
	\begin{pgfonlayer}{edgelayer}
		\draw [style=new edge style 0] (1) to (2);
		\draw [style=new edge style 3, bend left=15] (2) to (0);
		\draw [style=new edge style 2, bend right=15] (2) to (3);
		\draw [style=new edge style 3, bend left=15] (1) to (3);
		\draw [style=new edge style 1, bend right=45, looseness=1.50] (0) to (5);
		\draw [style=new edge style 0, bend left=15] (0) to (1);
		\draw [style=new edge style 1, bend right] (4) to (1);
		\draw [style=new edge style 3, in=-45, out=45] (3) to (4);
		\draw [style=new edge style 2, bend right=60, looseness=1.25] (5) to (0);
	\end{pgfonlayer}
	\end{tikzpicture}
\caption{$\Gamma_A(H)$ and its spanning tree $L_0$}\label{Gamma(H) and L}
\end{figure}
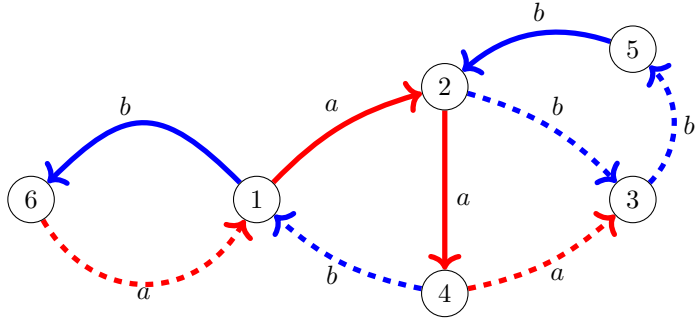

%
%
%			
%Since $B_{T_0}\subseteq B_{T}$, so $L(\Delta)$ is a Schreier free factor of $H$.
%
%\begin{figure}[H]
%	\centering
%	\begin{tikzpicture}
%	\begin{pgfonlayer}{nodelayer}
%		\node [style=new style 0] (0) at (0, 0) {1};
%		\node [style=new style 0] (1) at (2.5, 1.5) {2};
%		\node [style=new style 0] (2) at (2.5, -1.25) {4};
%		\node [style=new style 0] (3) at (5, 0) {3};
%		\node [style=none] (8) at (1, -1) {$b$};
%		\node [style=none] (9) at (1, 1.25) {$a$};
%		\node [style=none] (10) at (4, 1.25) {$b$};
%		\node [style=none] (11) at (4, -1) {$a$};
%		\node [style=none] (14) at (2.75, 0) {$a$};
%	\end{pgfonlayer}
%	\begin{pgfonlayer}{edgelayer}
%		\draw [style=new edge style 0] (1) to (2);
%		\draw [style=new edge style 3, bend left=15] (2) to (0);
%		\draw [style=new edge style 2, bend right=15] (2) to (3);
%		\draw [style=new edge style 3, bend left=15] (1) to (3);
%		\draw [style=new edge style 0, bend left=15] (0) to (1);
%	\end{pgfonlayer}
%	\end{tikzpicture}
%\caption{$\Gamma_A(H)$ and its spanning tree $\T_0$}\label{Delta and T0}
%\end{figure}
%
%Dotted line here depicts a spanning tree $\T_0$ of $\Delta$, and $\T_0$ corresponds to $T_0=\{1,b^{-1},b^{-1}a,b^{-1}ab^{-1}\}$, a Schreier transversal of $L(\Delta)$ in $F(A)$.
%
%The Schreier basis of $L(\Delta)$ is $B_{T_0}=\{ aba^{-1}b,b^{-1}a^{-1}ba^{-1}b \}$.
%

\end{exam}

\begin{lem}\label{Schf}
Let $H$ be a finitely generated subgroup of $F(A)$ represented by a reduced $A$-labeled graph $\Gamma_A(H)$. Suppose $\Delta$ is a reduced $A$-labeled subgraph of $\Gamma_A(H)$, then the subgroup $L(\Delta)$ associated with $\Delta$ is a Schreier free factor of $H$.
\end{lem}
\begin{proof}
See \cite{KM02} Lemma 6.1 and Proposition 6.2.
\end{proof}

\subsubsection{Morphisms between reduced $A$-labeled graphs}

\begin{defn}
A morphism $\phi$ between two reduced $A$-labeled graphs $\Gamma$ and $\Delta$ is a map from $\Gamma$ to $\Delta$ that preserves the structure of $A$-labeled graphs. Namely, $\phi$ sends the base vertex to the base vertex, and vertices to vertices, and edges to edges such that whenever $\Gamma$ has an edge $e$ labeled by $a\in A$ from $o_e$ to $t_e$, then $\phi(e)$ is an edge labeled by $a$ from $\phi(o_e)$ to $\phi(t_e)$ in $\Delta$.
\end{defn}

It is easy to observe that a morphism of reduced $A$-labeled graphs is necessarily locally injective (an immersion in \cite{Sta83}), namely for each vertex $v$ of $\Gamma$, different edges with the same origin (resp. the same terminus) $v$ have different images. The following properties are well known and proved in \cite{KM02} (or \cite{MVW01}).

\begin{prop}\label{morbasic}
Let $H, K$ be two finitely generated subgroups of $F(A)$. Then

(1) A morphism $\Gamma_A(H)\to \Gamma_A(K)$ exists if and only if $H\leq K$.

(2) If a morphism $\Gamma_A(H)\to \Gamma_A(K)$ exists, it is unique. We denote it by $\phi_{H,K}$.

(3) If a morphism $\Gamma_A(H)\to \Gamma_A(K)$ is injective, namely $\phi_{H,K}$  is one-to-one (if and only if it is one-to-one on vertices), then $H$ is a free factor of $K$.
\end{prop}

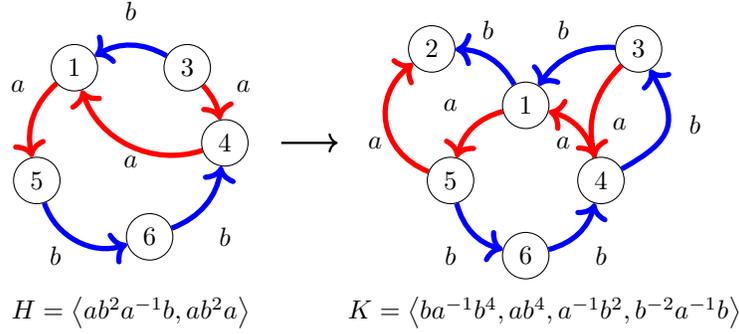
\begin{figure}[H]
	\centering
	\begin{tikzpicture}
		\begin{pgfonlayer}{nodelayer}
			\node [style=new style 0] (0) at (-0.5, 0.5) {1};
			\node [style=new style 0] (1) at (-1, -1) {5};
			\node [style=new style 0] (3) at (1, 0.5) {3};
			\node [style=new style 0] (4) at (1.5, -0.5) {4};
			\node [style=new style 0] (6) at (0.5, -1.75) {6};
			\node [style=none] (8) at (-0.75, -2) {$b$};
			\node [style=none] (9) at (1.5, -1.75) {$b$};
			\node [style=none] (11) at (-1.25, 0.25) {$a$};
			\node [style=none] (13) at (0.25, 1.25) {$b$};
			\node [style=none] (14) at (1.75, 0.25) {$a$};
			\node [style=new style 0] (17) at (5.5, 0) {1};
			\node [style=new style 0] (18) at (4.5, -1) {5};
			\node [style=new style 0] (19) at (7, 0.75) {3};
			\node [style=new style 0] (20) at (6.5, -1) {4};
			\node [style=new style 0] (21) at (5.5, -2) {6};
			\node [style=new style 0] (22) at (4.25, 0.75) {2};
			\node [style=none] (23) at (4.5, -2) {$b$};
			\node [style=none] (24) at (6.5, -2) {$b$};
			\node [style=none] (25) at (3.5, -0.5) {$a$};
			\node [style=none] (26) at (4.5, 0) {$a$};
			\node [style=none] (27) at (5, 1) {$b$};
			\node [style=none] (28) at (6, 1) {$b$};
			\node [style=none] (29) at (6.75, -0.25) {$a$};
			\node [style=none] (30) at (6, -0.5) {$a$};
			\node [style=none] (31) at (7.75, -0.25) {$b$};
			\node [style=none] (32) at (2.25, -0.5) {};
			\node [style=none] (33) at (3, -0.5) {};
			\node [style=none] (34) at (0.25, -2.75) {$H=\left \langle ab^2a^{-1}b,ab^2a \right \rangle$};
			\node [style=none] (35) at (5.75, -2.75) {$K=\left \langle ba^{-1}b^4,ab^4,a^{-1}b^2,b^{-2}a^{-1}b \right \rangle$};
			\node [style=none] (36) at (0.25, -0.75) {$a$};
		\end{pgfonlayer}
		\begin{pgfonlayer}{edgelayer}
			\draw [style=new edge style 0, bend right] (0) to (1);
			\draw [style=new edge style 0, bend left=15] (3) to (4);
			\draw [style=new edge style 1, bend right] (3) to (0);
			\draw [style=new edge style 1, bend right=45] (1) to (6);
			\draw [style=new edge style 1, bend right] (6) to (4);
			\draw [style=new edge style 0, bend left=60, looseness=1.25] (18) to (22);
			\draw [style=new edge style 0, bend right] (17) to (18);
			\draw [style=new edge style 0, bend right] (20) to (17);
			\draw [style=new edge style 0, bend right] (19) to (20);
			\draw [style=new edge style 1, bend right] (19) to (17);
			\draw [style=new edge style 1, bend right] (17) to (22);
			\draw [style=new edge style 1, bend right] (18) to (21);
			\draw [style=new edge style 1, bend right] (21) to (20);
			\draw [style=new edge style 1, bend right=45, looseness=1.50] (20) to (19);
			\draw [style=new edge style 4] (32.center) to (33.center);
			\draw [style=new edge style 0, bend right=315] (4) to (0);
		\end{pgfonlayer}
	\end{tikzpicture}
	
	\caption{Injective morphism between reduced $A$-labeled graphs}\label{fig injective morphism}
	\end{figure}

\begin{rem}\label{mSchf}
If the morphism $\Gamma_A(H)\to \Gamma_A(K)$ is injective, namely the image of $\Gamma_A(H)$ is a reduced $A$-labeled subgraph of $\Gamma_A(K)$, then $H$ is a Schreier free factor of $K$ by Lemma \ref{Schf}.
\end{rem}

\begin{defn}
Let $H, K$ be two finitely generated subgroups of $F(A)$. If $\phi_{H,K}:\Gamma_A(H)\to \Gamma_A(K)$ is surjective (both on vertices and on edges), then we say that $K$ is an overgroup of $H$ (with respect to $A$). The set of all overgroups of $H$ is called the \emph{$A$-fringe} of $H$, denoted by $\mathcal{O}_A(H)$.
\end{defn}

\begin{figure}[H]
\centering
	\begin{tikzpicture}
	\begin{pgfonlayer}{nodelayer}
		\node [style=new style 0] (0) at (-5, 1.75) {2};
		\node [style=new style 0] (1) at (-5.75, 0.5) {1};
		\node [style=new style 0] (2) at (-3.25, 1.5) {3};
		\node [style=new style 0] (3) at (-5, -0.75) {5};
		\node [style=new style 0] (4) at (-3.25, -0.5) {4};
		\node [style=none] (5) at (-6, 1.5) {$a$};
		\node [style=none] (6) at (-6, -0.5) {$a$};
		\node [style=none] (7) at (-3.75, -1.25) {$b$};
		\node [style=none] (8) at (-2.5, 0.5) {$a$};
		\node [style=none] (9) at (-4, 2.25) {$b$};
		\node [style=new style 0] (10) at (0, 0.5) {1};
		\node [style=new style 0] (11) at (1.5, 0.5) {2,5};
		\node [style=new style 0] (12) at (3, 0.5) {3,4};
		\node [style=none] (13) at (0.75, -0.5) {$a$};
		\node [style=none] (14) at (2.25, -0.5) {$b$};
		\node [style=none] (15) at (0.75, 1.5) {$a$};
		\node [style=none] (16) at (2.25, 1.5) {$b$};
		\node [style=none] (17) at (-1.5, 0.5) {};
		\node [style=none] (18) at (-0.75, 0.5) {};
		\node [style=none] (19) at (-4.5, -1.75) {$H=\left \langle ab^2a^{-1},aba^2b^{-1}a^{-1},ababa \right \rangle$};
		\node [style=none] (20) at (2, -1.75) {$K=\left \langle a^2,ab^2a,ababa \right \rangle$};
		\node [style=none] (26) at (4, 0.5) {$a$};
		\node [style=none] (27) at (-4.25, 0.75) {$b$};
		\node [style=none] (28) at (-4, 0) {$a$};
	\end{pgfonlayer}
	\begin{pgfonlayer}{edgelayer}
		\draw [style=new edge style 0, bend left] (1) to (0);
		\draw [style=new edge style 0, bend left] (3) to (1);
		\draw [style=new edge style 1, bend left] (0) to (2);
		\draw [style=new edge style 1, bend left] (4) to (3);
		\draw [style=new edge style 1, bend left=60, looseness=1.25] (11) to (12);
		\draw [style=new edge style 0, bend left=60, looseness=1.25] (11) to (10);
		\draw [style=new edge style 4] (17.center) to (18.center);
		\draw [style=new edge style 0, bend left=45] (2) to (4);
		\draw [style=new edge style 0, bend left=60, looseness=1.25] (10) to (11);
		\draw [style=new edge style 1, bend left=60, looseness=1.25] (12) to (11);
		\draw [style=new edge style 0, in=-45, out=45, loop] (12) to ();
		\draw [style=new edge style 1, bend right=315] (2) to (0);
		\draw [style=new edge style 0, bend left] (4) to (2);
	\end{pgfonlayer}
	\end{tikzpicture}	
\caption{Surjective morphism between reduced $A$-labeled graphs}\label{fig surjective morphism}
\end{figure}
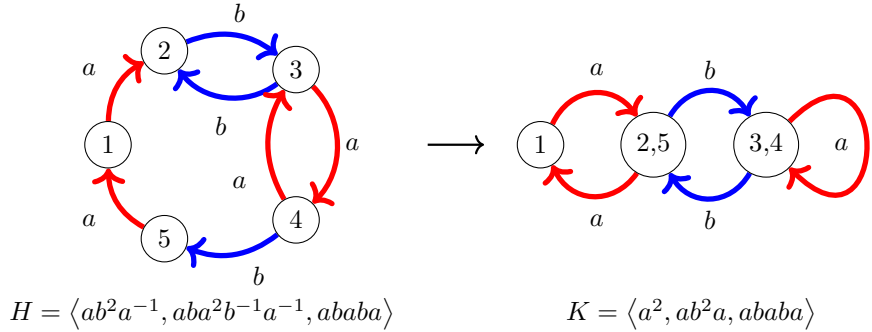

Since a finite reduced $A$-labeled graph can only be mapped onto finitely many reduced $A$-labeled graphs, we have

\begin{prop}
Let $H$ be a finitely generated subgroup of $F(A)$, then $\mathcal{O}_A(H)$ is a finite set.
\end{prop}

\subsection{$\var{V}$-preclosed subgroups}\label{sp4}

\begin{defn}[\cite{Weil98}, Section 2.1]
Let $\var{V}$ be a variety. A subgroup $H$ of $F(A)$ is said to be $\var{V}$-preclosed if there exists a $\var{V}$-open subgroup $K$ containing $H$ such that the natural morphism $\phi_{H,K}: \Gamma_A(H)\to \Gamma_A(K)$ is injective. The intersection of all $\var{V}$-preclosed subgroups that containing $H$ is called the $\var{V}$-preclosure of $H$.% and denoted by $Pcl_\var{V}(H)$.
\end{defn}
%
%The notion of $\var{V}$-preclosed subgroup is not an algebraic notion and it depends on the chosen basis $A$.

%\begin{lem}\label{cpclf}
%Let $\var{V}$ be a variety. Then
%
%(1) a $\var{V}$-closed subgroup is $\var{V}$-preclosed;
%
%(2) a $\var{V}$-preclosed subgroup is a free factor of a $\var{V}$-open subgroup.
%\end{lem}
%\begin{proof}
%Suppose $H$ is $\var{V}$-closed, then by it is a subgroup of any $\var{V}$-open subgroup $K$ containing $H$. The image of the natural morphism  $\phi_{H,K}: \Gamma_A(H)\to \Gamma_A(K)$ is
%\end{proof}

%\begin{prop}
%Let $\var{V}$ be a variety and $H$ a $\var{V}$-closed subgroup of $F(A)$, then $H$ is $\var{V}$-preclosed
%\end{prop}

A $\var{V}$-preclosed subgroup is a free factor of a $\var{V}$-open subgroup by Proposition \ref{morbasic} (3), but the converse does not hold in general, see \cite{Weil98} Example 2.3. This is because the notion of $\var{V}$-preclosed subgroup is not an algebraic notion and it depends on the chosen basis $A$.

In \cite{Weil98} Section 2, the author discussed the relation between $\var{V}$-closed subgroups and $\var{V}$-preclosed subgroups, and proved
that
\begin{prop}[\cite{Weil98}, Proposition 2.8]\label{ctopc}
A  $\var{V}$-closed subgroup of the free group $F$ is $\var{V}$-preclosed.
\end{prop}

\begin{prop}[\cite{Weil98}, Section 2.1.2]\label{preover}
Let $H$ be a finitely generated subgroup of $F(A)$, then the $\var{V}$-preclosure of $H$ is an overgroup of $H$.
\end{prop}

The author also found that under some assumptions about the variety $\var{V}$, these two notions coincide.

\begin{defn}
A variety $\var{V}$ is called \emph{extension-closed} if for a finite group $G$ whenever a normal subgroup $K$ of $G$  and $G/K$ lie in $\var{V}$, then $G$ is in $\var{V}$.
\end{defn}

\begin{prop}[\cite{Weil98}, Proposion 2.9]\label{cpclf}
Assume $\var{V}$ is an extension-closed variety and let $H$ be a finitely generated subgroup of the free group. Then the following are equivalent:

(1) $H$ is $\var{V}$-closed;

(2) $H$ is $\var{V}$-preclosed;

(3) $H$ is a free factor of a $\var{V}$-open subgroup.
\end{prop}

%The variety $\var{H}_p=\var{G}_p* \var{Ab}_{p-1}$ ($p$ is prime) is not extenstion-closed, but a similar result like Proposition \ref{cpclf}  holds for $\var{H}_p$ since  it is a variety with the M. Hall property or shortly called  a Hall variety.

A similar result like Proposition \ref{cpclf} also holds  for a variety with the M. Hall property or shortly called  a Hall variety defined below.

\begin{defn}[\cite{AS}, Section 3.5]\label{dhall}
A variety $\var{V}$ is said to be a Hall variety if for each finite alphabet $A$ and each $\var{V}$-open subgroup $K$ of $F(A)$, each Schreier free factor of $K$ is $\var{V}$-closed in $F(A)$.
\end{defn}

\begin{prop}\label{cpclfhall}
Assume $\var{V}$ is a Hall variety and let $H$ be a finitely generated subgroup of the free group $F(A)$. Then the following are equivalent:

(1) $H$ is $\var{V}$-closed;

(2) $H$ is $\var{V}$-preclosed;

(3) $H$ is a Schreier free factor of a $\var{V}$-open subgroup.
\end{prop}
\begin{proof}
(3)$\Rightarrow$(1) is by Definition \ref{dhall}. (1)$\Rightarrow$(2) is by Proposition \ref{ctopc}. (2)$\Rightarrow$(3) is by Remark \ref{mSchf}.
\end{proof}

\section{Supersolvable closures of subgroups}

\subsection{$\var{H}_p$-closures of subgroups}\label{t1}

\subsubsection{$\var{H}_p$-closures are overgroups}

In \cite{AS}, the authors gave a complete classification of varieties of supersolvable groups those are Hall varieties. They proved

\begin{thm}[\cite{AS}, Theorem 1.1]
For a variety $\var{V}$ of finite supersolvable groups, then $\var{V}$ is a Hall variety if and only if $\var{V}=\var{G}_p * \var{Ab}_d$ for some prime $p$ and some positive integer $d~|~p-1$.
\end{thm}

\begin{cor}\label{hpclo}
Let $H$ be a finitely generated subgroup of $F(A)$, then the $\var{H}_p$-closure of $H$ is an overgroup of $H$, namely $Cl_{\var{H}_p}(H)$ lies in $\mathcal{O}_A(H)$.
\end{cor}
\begin{proof}
By Proposition \ref{cpclfhall} $Cl_{\var{H}_p}(H)$ is exactly the $\var{H}_p$-preclosure of $H$ and by Proposition \ref{preover} it is an overgroup of $H$.
\end{proof}

\subsubsection{Deciding $\var{H}_p$-denseness}

%\begin{lem}
%Let $G$ be a group in the variety $\var{H}_p$, then $G/\Phi(G)$ is in the variety $\var{Ab}_p*\var{Ab}_{p-1}$ where $\Phi(G)$ is the Frattini subgroup of $G$.
%\end{lem}

\begin{defn}
We call $(G,S)$ an \emph{$n$-marked group} if $S$ is a generating set of $G$ containing $n$ elements. A homomorphism between two $n$-marked groups $(G, s_1,...,s_n)$, $(G', s'_1,...,s'_n)$ is a group homomorphism $\phi: G\to G'$ such that $\phi(s_i)=s'_i$ for each $i=1,...,n$.
\end{defn}

\begin{defn}%[free object]
Let $\var{V}$ be a variety and $F=F(A)$ the free group over a finite alphabet $A$. We call $F_\var{V}(A)$ a \emph{finite free object} in $\var{V}$ over $A$ if $F_\var{V}(A)$ is in $\var{V}$ and
there exists a homomorphism $\sigma: (F,A)\to (F_\var{V}(A), \sigma(A))$  called the canonical projection satisfying the following universal property: for any $\psi: (F,A)\to (G,S)\in\var{V}$, there exists $\hat{\psi}: (F_\var{V}(A), \sigma(A))\to (G,S)$ such that the following diagram commutes.
\end{defn}

\begin{center}
\begin{tikzcd}
		{(F(A), A)} && {(G, S)} \\
		\\
		{(F_\var{V}(A),\sigma(A))}
		\arrow["\sigma"', from=1-1, to=3-1]
		\arrow["\psi", from=1-1, to=1-3]
		\arrow["{\hat{\psi}}"', dashed, from=3-1, to=1-3]
\end{tikzcd}
\end{center}

\begin{rem}
In \cite{MSW01} Section 2, the authors described the free object over $A$ with respect to $\var{V}$. When $\var{V}$ has a finite free object over $A$, it is exactly the $F_\var{V}(A)$ defined above.
\end{rem}

\begin{exam}\label{abm}
For the variety $\var{Ab}_m$, $F_{\var{Ab}_m}(A)$ is isomorphic to $(\Z/m\Z)^n$ (usually denoted by $(\Z/m\Z)^A$) where $n$ is the cardinality of $A$. The kernel of the canonical projection $\sigma: (F,A)\to (F_{\var{Ab}_m}(A), \sigma(A))$ is $[F,F]F^m$.
\end{exam}

In \cite{MSW01}, Margolis-Sapir-Weil proved the following Lemma  \ref{gpabp} and Corollary \ref{pdec} stating that $\var{G}_p$-denseness is equivalent to $\var{Ab}_p$-denseness and decidable. %This allows them to give an algorithm to compute the $\var{G}_p$-closure of a finitely generated subgroup.

\begin{lem}[\cite{MSW01}, Corollary 3.2]\label{gpabp}
Let $H$ be a finitely generated subgroup of $F=F(A)$, and let
$\sigma: F(A)\to (\Z/p\Z)^A$ be the canonical projection. The following are equivalent.

(1) $H$ is $\var{G}_p$-dense in $F$;

(2) $H$ is $\var{Ab}_p$-dense in $F$;

(3) $\sigma^{-1}\sigma(H)=F$;

(4) $\sigma(H)=(\Z/p\Z)^A$.
\end{lem}

\begin{rem}\label{pdense}
Let $N$ be the kernel of $\sigma$, that is $N=[F,F]F^p$, then $H$ satisfies the above condition (3) if and only if $H\cdot N=F$.
\end{rem}

\begin{cor}[\cite{MSW01}, Corollary 3.4]\label{pdec}
It is decidable whether $H$ is $\var{G}_p$-dense.
\end{cor}

\begin{cor}[\cite{MSW01} Section 3.2 or \cite{RZ94} Theorem 4.2]\label{gpc}
Let $H$ be a finitely generated subgroup of $F$, then the $\var{G}_p$-closure of $H$ is computable.
\end{cor}

There is a similar relationship between $\var{H}_p$ and its subvariety $\var{Ab}_p*\var{Ab}_{p-1}$ like that in Lemma \ref{gpabp}. We will prove  $\var{H}_p$-denseness is equivalent to $\var{Ab}_p*\var{Ab}_{p-1}$-denseness (Theorem \ref{hpdense}) and decidable (Theorem \ref{hpd}).

A variety $\var{V}$ is called \emph{locally finite} if $F_\var{V}(A)$ exists for any finite alphabet $A$. For instance, the variety $\var{Ab}_m$ in Example \ref{abm}  is locally finite. The product of two locally finite varieties is also locally finite.

\begin{lem}\label{kernel}
The variety $\var{Ab}_p*\var{Ab}_{p-1}$ is locally finite. Moreover, the kernel of the canonical projection $\sigma: (F(A),A)\to (F_{\var{Ab}_p*\var{Ab}_{p-1}}(A),\sigma(A))$ is $[N_{p-1},N_{p-1}]N^p_{p-1}$ where $N_{p-1}=[F,F]F^{p-1}$.
\end{lem}
\begin{proof}
See \cite{Neu67} 21.12, 21.13 or \cite{AS} page 871.
\end{proof}

\begin{lem}\label{atot}
Let $H$ be a finitely generated subgroup of $F$ and let $N_{p-1}=[F,F]F^{p-1}$ where $p$ is a prime number. Then $Cl_{\var{Ab}_p*\var{Ab}_{p-1}}(H)=H\cdot [N_{p-1},N_{p-1}]N^p_{p-1}$.
\end{lem}
\begin{proof}
Let $A$ be a basis of $F$, then by Lemma \ref{kernel} the kernel of the canonical projection $\sigma: (F(A),A)\to (F_{\var{Ab}_p*\var{Ab}_{p-1}}(A),\sigma(A))$ is $[N_{p-1},N_{p-1}]N^p_{p-1}$.

Note that $F_{\var{Ab}_p*\var{Ab}_{p-1}}(A)$ is in the variety $\var{Ab}_p*\var{Ab}_{p-1}$ and $\sigma$ is in the set $\Phi_{\var{Ab}_p*\var{Ab}_{p-1}}$ of all homomorphisms from $F$ to groups in $\var{Ab}_p*\var{Ab}_{p-1}$, we have $$Cl_{\var{Ab}_p*\var{Ab}_{p-1}}(H)=\bigcap_{\phi\in \Phi_{\var{Ab}_p*\var{Ab}_{p-1}}} \phi^{-1}(\phi(H))\leq \sigma^{-1}(\sigma(H)).$$

For any homomorphism $\phi\in \Phi_{\var{Ab}_p*\var{Ab}_{p-1}}$, denote also by $\phi$ the induced surjective homomorphism $(F,A)\to (\phi(F),\phi(A))$. Since $F_{\var{Ab}_p*\var{Ab}_{p-1}}(A)$ is the free object in $\var{Ab}_p*\var{Ab}_{p-1}$ over $A$, there exists $\hat{\phi}: (F_{\var{Ab}_p*\var{Ab}_{p-1}}(A), \sigma(A))\to (\phi(F),\phi(A))$ such that $\phi=\hat{\phi}\circ \sigma$. Thus
$\sigma^{-1}(\sigma(H))\leq \sigma^{-1}(\hat{\phi}^{-1}(\hat{\phi}\circ\sigma(H)))= \phi^{-1}(\phi(H))$, then $$\sigma^{-1} (\sigma(H))\leq \displaystyle \bigcap_{\phi\in \Phi_{\var{Ab}_p*\var{Ab}_{p-1}}} \phi^{-1}(\phi(H))=Cl_{\var{Ab}_p*\var{Ab}_{p-1}}(H).$$

So we have $Cl_{\var{Ab}_p*\var{Ab}_{p-1}}(H)=\sigma^{-1} (\sigma(H))=H\cdot [N_{p-1},N_{p-1}]N^p_{p-1}$.
\end{proof}

%So $H$ is $\var{Ab}_p*\var{Ab}_{p-1}$-dense in $F$ if and only if $H\cdot [N_{p-1},N_{p-1}]N^p_{p-1}=\sigma^{-1} (\sigma(H))=F$.

\begin{lem}\label{hptoa}
Let $H$ be a finitely generated subgroup of $F$, then $H$ is $\var{H}_p$-dense in $F$ if and only if $H$ is $\var{Ab}_p*\var{Ab}_{p-1}$-dense in $F$.
\end{lem}
\begin{proof}
Suppose $H$ is $\var{H}_p$-dense in $F$. The variety $\var{Ab}_p*\var{Ab}_{p-1}$ is a subvariety of $\var{H}_p$, then by Lemma \ref{subd} $H$ is $\var{Ab}_p*\var{Ab}_{p-1}$-dense in $F$.

Conversely, assume $H$ is $\var{Ab}_p*\var{Ab}_{p-1}$-dense but not $\var{H}_p$-dense in $F$, then by Corollary \ref{dense} there exists a surjective homomorphism $\phi: F\to G$ for some $G$ in the variety $\var{H}_p=\var{G}_p* \var{Ab}_{p-1}$ such that $\phi(H)\lneqq \phi(F)=G$.

Suppose $M$ is a maximal subgroup of $G$ containing $\phi(H)$, then  for the Frattini subgroup $\Phi(G)$ of $G$ which is the intersection of all maximal subgroups of $G$, we have  $\phi(H)\cdot\Phi(G)\leq M$.

Since $p$ is coprime with $p-1$, by Schur-Zassenhaus theorem $G$ is a semidirect product $P\rtimes B$ for some $p$-group $P\in \var{G}_p$ and $B\in \var{Ab}_{p-1}$. Note that $P/\Phi(P)$ is an abelian $p$-group, that is $P/\Phi(P)\in \var{Ab}_p$, then $P\rtimes B/\Phi(P)\cong (P/\Phi(P))\rtimes B$ lies in $\var{Ab}_p*\var{Ab}_{p-1}$. Since $\Phi(P)\leq \Phi(G)$, $G/\Phi(G)$ is a quotient of $P\rtimes B/\Phi(P)$, and it lies in the variety $\var{Ab}_p*\var{Ab}_{p-1}$ also.

Let $\tau: G\to G/\Phi(G)$ be the natural quotient map, then $$\tau\circ\phi(H)=(\phi(H)\cdot \Phi(G))/\Phi(G)\leq M/\Phi(G)$$ which is proper in $\tau\circ\phi(F)=G/\Phi(G)$, so $H$ is not $\var{Ab}_p*\var{Ab}_{p-1}$-dense in $F$, contrary to the assumption.
\end{proof}

%Denote by $\sigma$ the composition  of the abelianization map $F\to \Z^r$ with $\mod p-1$ quotient map  $\Z^r\to \Z_{p-1}^r$ where $r$ is the rank of $F$. Note that $N_{p-1}$ is the kernel of $\sigma$.
%
%Denote by $\tau$ the natural quotient map from $F$ to $F/[N_{p-1},N_{p-1}]N^p_{p-1}$.
%
%$\tau(H)=\tau(F)$
%
%
%Suppose $H$ is $\var{Ab}_p*\var{Ab}_{p-1}$-dense in $F$, since $\var{Ab}_{p-1}$ is a subvariety of $\var{Ab}_p*\var{Ab}_{p-1}$, $H$ is $\var{Ab}_{p-1}$ in $F$. If $H\cap N_{p-1}$ is not $\var{Ab}_{p}$-dense in $N_{p-1}$, then there exists $P\in \var{Ab}_{p}$ and a surjective homomorphism $\phi: N_{p-1}\to P$ such that $\phi(H\cap N_{p-1})\neq \phi(N_{p-1})=P$.

\begin{thm}\label{hpdense}
Let $H$ be a finitely generated subgroup of $F$, then the following are equivalent:

(1) $H$ is $\var{H}_p$-dense in $F$;

(2) $H$ is $\var{Ab}_p*\var{Ab}_{p-1}$-dense in $F$;

(3) $H\cdot [N_{p-1},N_{p-1}]N^p_{p-1}=F$ where $N_{p-1}=[F,F]F^{p-1}$;

(4) $H$ is $\var{Ab}_{p-1}$-dense in $F$ and $H\cap N_{p-1}$ is $\var{Ab}_{p}$-dense in $N_{p-1}$.
\end{thm}
\begin{proof}
By Lemma \ref{atot} and \ref{hptoa} and we only need to prove (3)$\Leftrightarrow$(4).

%$$H\cdot [N_{p-1},N_{p-1}]N^p_{p-1}=F$$ if and only if $H$ satisfies the above condition (3).

Suppose $H\cdot [N_{p-1},N_{p-1}]N^p_{p-1}=F$, then $H\cdot N_{p-1}=F$ since $[N_{p-1},N_{p-1}]N^p_{p-1}$ is a subgroup of $N_{p-1}$, and we have $$(H\cap N_{p-1})\cdot  [N_{p-1},N_{p-1}]N^p_{p-1}=(H\cdot [N_{p-1},N_{p-1}]N^p_{p-1})\cap N_{p-1}=F\cap N_{p-1}=N_{p-1}.$$ So by Remark \ref{pdense}, $H$ is $\var{Ab}_{p-1}$-dense in $F$ and $H\cap N_{p-1}$ is $\var{Ab}_{p}$-dense in $N_{p-1}$.

Conversely, assume $H$ is $\var{Ab}_{p-1}$-dense in $F$ and $H\cap N_{p-1}$ is $\var{Ab}_{p}$-dense in $N_{p-1}$, then by Remark \ref{pdense}, $H\cdot N_{p-1}=F$ and $(H\cap N_{p-1})\cdot  [N_{p-1},N_{p-1}]N^p_{p-1}=N_{p-1}$ and we have $$F=H\cdot N_{p-1}=H\cdot ((H\cap N_{p-1})\cdot  [N_{p-1},N_{p-1}]N^p_{p-1})=H\cdot [N_{p-1},N_{p-1}]N^p_{p-1}.$$
\end{proof}

\begin{lem}\label{abd}
Let $H$ be a finitely generated subgroup of $F$, then $H$ is  $\var{Ab}_d$-dense in $F$ if and only if for each prime factor $p$ of $d$, $H$ is  $\var{Ab}_p$-dense in $F$.
\end{lem}
\begin{proof}
Suppose $H$ is  $\var{Ab}_d$-dense in $F$, since $\var{Ab}_p$ is a subvariety of $\var{Ab}_d$, $H$ is  $\var{Ab}_p$-dense by Lemma \ref{subd}.

Conversely, assume $H$ is  $\var{Ab}_p$-dense for each prime $p$ divides $d$ but not $\var{Ab}_d$-dense, then by Corollary \ref{dense} there exist $G\in \var{Ab}_d$ and a surjective homomorphism $\phi: F\to G$ such that $\phi(H)\lneqq \phi(F)=G$. The quotient group $G/\phi(H)$ is a nontrivial abelian group and so can project onto  a cyclic group $\Z/p\Z$ for some prime $p$ divides $d$. Then for the composition $\psi: F\to G\to G/\phi(H)\to \Z/p\Z$, $1=\psi(H)\lneqq \psi(F)=\Z/p\Z$, contrary to the assumption that $H$ is  $\var{Ab}_p$-dense.
\end{proof}

By Lemma \ref{abd} and Corollary \ref{pdec}, it is decidable whether $H$ is  $\var{Ab}_{p-1}$-dense in $F$ and $H\cap N_{p-1}$ is $\var{Ab}_{p}$-dense in $N_{p-1}$, thus by Theorem \ref{hpdense}  we have

\begin{cor}\label{hpd}
Let $H$ be a finitely generated subgroup of $F$, then it is decidable whether $H$ is $\var{H}_p$-dense.
\end{cor}

%The variety $\var{H}_p=\var{G}_p* \var{Ab}_{p-1}$ ($p$ is prime) is not extenstion-closed, but a similar result like Proposition \ref{cpclf}  holds for $\var{H}_p$ since  it is a variety with the M. Hall property or shortly called  a Hall variety.

\subsubsection{Computing $\var{H}_p$-closures} For two finitely generated subgroups $H,K$ of $F$ such that $H\leq K$, let $K$ be endowed with the pro-$\var{V}$ topology, we denote the $\var{V}$-closure of $H$ in $K$ by $Cl_{\var{V}}(H,K)$.

The pro-$\var{V}$ topology of $K$ is usually not the restriction to $K$ of the pro-$\var{V}$ topology of $F$, but when $\var{V}$ is an extension-closed variety, then for an open subgroup $K$, these two topology coincide (\cite{MSW01}, Proposition 1.6). This is the key for the authors of \cite{MSW01} to construct an algorithm to compute the $\var{G}_p$-closure of a subgroup.

In \cite{AS}, Auinger-Steinberg proved that the $\var{V}$-closure of a finitely generated subgroup is computable for certain product variety $\var{V}=\var{U}*\var{W}$  which can be applied to the case that $\var{V}=\var{H}_p=\var{G}_p*\var{Ab}_{p-1}$.

\begin{prop}[\cite{AS}, Proposition 4.4]
Suppose that $\var{U}$ and $\var{W}$ are varieties of groups with $\var{U}$ having computable closures and $\var{W}$ being locally finite. Suppose, moreover, there is an algorithm that on input a finite set $A$ outputs generators of the kernel $K$ of the canonical projection $\rho: F(A)\to F_\var{W}(A)$. Then $\var{U}*\var{W}$ has computable closures.
\end{prop}

\begin{rem}
In fact, the authors proved $$Cl_{\var{U}*\var{W}}(H)=\bigcup_{t\in T} Cl_{\var{U}}(H\cap K, K)t,$$
where $T$ is a Schreier transversal for the right cosets of $H\cap K$ in $H$ which is computable under the assumption in the above proposition. %The $\var{H}_p$-closure of $H$ is $\bigcup_{t\in T} Cl_{\var{G}_p}(H\cap N_{p-1}, N_{p-1})t$ where $T$ and $Cl_{\var{G}_p}(H\cap N_{p-1}, N_{p-1})$ are both computable.
\end{rem}

For $\var{U}=\var{G}_p$ and $\var{W}=\var{Ab}_{p-1}$, the $\var{G}_p$-closure can be computed by Corollary \ref{gpc}, $\var{W}$ is locally finite and the kernel of the canonical projection $\rho: F(A)\to F_\var{W}(A)$ is $N_{p-1}=[F,F]F^{p-1}$ mentioned before whose generators can be expressed as words over the alphabet $A$ by an algorithm. We have the following

\begin{cor}
Let $H$ be a finitely generated subgroup of $F$, then the $\var{H}_p$-closure of $H$ is computable.
\end{cor}
%
%\begin{cor}
%Let $H$ be a finitely subgroup of $F$ such that $H\leq N_{p-1}$ where $N_{p-1}=[F,F]F^{p-1}$ and $p$ is prime, then $Cl_{\var{H}_p}(H)=Cl_{\var{G}_p}(H,N_{p-1})$.
%\end{cor}
%\begin{proof}
%Let $K=Cl_{\var{H}_p}(H)$ and $L=Cl_{\var{G}_p}(H,N_{p-1})\leq N_{p-1}$. Note that $K$ is also a subgroup of $N_{p-1}$ because $N_{p-1}$ is a $\var{H}_p$-closed subgroup containing $H$.
%
% If there exists some $l\in L$ such that $l\notin K$, then there is a homomorphism $\phi: F \to G$ for some  $G\in\var{H}_p$ such that $\phi(l)\notin \phi(H)$. Note that $\phi|_{N_{p-1}}: N_{p-1}\to P$ where $P$ lies in $\var{G}_p$, we have $\phi|_{N_{p-1}}(l)=\phi(l)\notin\phi(H)= \phi|_{N_{p-1}}(H)$, contrary to the assumption that $l$ lies in $L$.
%
%Conversely, if there exists some $k\in K$ such that $k\notin L$, then there is a homomorphism $\psi: N_{p-1}\to P$ for some $P\in \var{G}_p$ such that $\psi(k)\notin \psi(H)$,
%\end{proof}
%
%\begin{thm}
%Let $H$ be a finitely subgroup of $F$, then the $\var{H}_p$-closure of $H$ is computable.
%\end{thm}

\subsection{Supersolvable closures of subgroups}\label{t2}
Let $\mathbb{P}$ be the set of all prime numbers, for a nonempty subset $\pi\subset \mathbb{P}$, denote by $\pi'$ the complement of $\pi$ in $\mathbb{P}$. Moreover, if $\pi$ consists of only one element $p$, we denote $\pi'$ by $p'$. We denote the set of all prime numbers dividing $n$ by $\pi(n)$ and $\pi(|G|)$ by $\pi(G)$ for a finite group $G$.

\begin{defn}[$\pi$-group]
Let $\pi$ be a nonempty subset of $\PP$. A finite group $G$ is called a $\pi$-group if $\pi(G)\subset \pi$.
\end{defn}

\begin{defn}[$\pi$-core]
For a nonempty subset $\pi\subset \PP$, the $\pi$-core $O_{\pi}(G)$ of $G$ is the unique largest normal $\pi$-subgroup of $G$. If $\pi$ consists of only one element $p$, we denote $O_{\{p\}}(G)$ by $O_p(G)$ which is called the $p$-core of $G$.
\end{defn}

\begin{rem}
The $p$-core $O_p(G)$ is the intersection of all Sylow $p$-subgroups of $G$. If $p\notin \pi(G)$, then $O_p(G)=1$ while $O_{p'}(G)=G$.
\end{rem}

\begin{lem}\label{ff}
Let $G$ be a finite group, then $\bigcap_{p\in\pi(G)}O_{p'}(G)$ is trivial.
\end{lem}
\begin{proof}
For each $p\in\pi(G)$, the order of $\bigcap_{p\in\pi(G)}O_{p'}(G)$ divides the order of $O_{p'}(G)$. It is coprime with each prime factor of $\pi(G)$, hence $\bigcap_{p\in\pi(G)}O_{p'}(G)$ is trivial.
\end{proof}

\begin{lem}\label{fff}
Let $M$  be a subgroup of a finite group $G$, then $M=\bigcap_{p\in\pi(G)}MO_{p'}(G)$.
\end{lem}
\begin{proof}
Suppose $\pi(G)=\{p_1,...,p_k\}$, we have  $|G|=p_1^{n_1}\cdots p_k^{n_k}$, $|M|=p_1^{m_1}\cdots p_k^{m_k}$ and $|O_{p'_i}(G)|=p_1^{\a_{i1}}\cdots p_i^{\a_{ii}}\cdots p_k^{\a_{ik}}$ for $i=1,...,k$. Note that $\a_{ii}=0$ since $p_i$ is coprime with $|O_{p'_i}(G)|$.

%$|M|$ divides the order of $\bigcap_{i=1}^k MO_{p'_i}(G)$ $M$ is a subgroup of $\bigcap_{i=1}^k MO_{p'_i}(G)$,  .
The order of $MO_{p'_i}(G)$ divides % is equal to $\frac{|M||O_{p'_i}(G)|}{|M\cap O_{p'_i}(G)|}$, and it
$$|M||O_{p'_i}(G)|=p_1^{m_1+\a_{i1}}\cdots p_{i-1}^{m_{i-1}+\a_{i,i-1}}p_i^{m_i}p_{i+1}^{m_{i+1}+\a_{i,i+1}}\cdots p_k^{m_k+\a_{ik}},$$
then the order of $\bigcap_{i=1}^k MO_{p'_i}(G)$ divides the greatest common divisor $$\gcd(|M||O_{p'_1}(G)|,\cdots,|M||O_{p'_k}(G)|)=p_1^{m_1}\cdots p_k^{m_k}=|M|.$$

Thus $M=\bigcap_{i=1}^k MO_{p'_i}(G)$ because $\bigcap_{i=1}^k MO_{p'_i}(G)$ has the same order with its subgroup $M$.
\end{proof}

\begin{defn}[$p$-nilpotent]
Assume $p$ is a prime number, a finite group $G$ is said to be $p$-nilpotent if for each chief factor $M/N$ of $G$, either $M/N$ is a $p'$-group or $M/N$ is in the center of $G/N$. For a finite group $G$, we denote the product of all normal $p$-nilpotent subgroups of $G$  by $F_{p}(G)$. It is a characteristic subgroup of $G$.
\end{defn}
%\begin{rem}
%A finite group $G$ is $p$-nilpotent if and only if $O_p(G/O_{p'}(G))=G$.
%\end{rem}

\begin{lem}[\cite{Guo}, Theorem 1.8.13 (2), page 41]\label{fpg}
Let $G$ be a finite group, then $F_p(G)/O_{p'}(G)=O_p(G/O_{p'}(G))$.
\end{lem}

\begin{lem}[\cite{Guo}, Example 4, page 98]\label{abp}
Let $G$ be a finite group,  then $G$ is supersolvable if and only if $G/F_p(G)$ lies in the variety $\var{Ab}_{p-1}$ for each $p\in \pi(G)$.
\end{lem}

\begin{lem}\label{fp}
Suppose $G$ is a finite supersolvable group, then $G/O_{p'}(G)$ lies in the variety $\var{H}_p$ for each $p\in \pi(G)$.
\end{lem}
\begin{proof}
%Let $F_{p}(G)$ be the product of all normal $p$-nilpotent subgroups of $G$.

We have an exact sequence of groups
$$1\To F_p(G)/O_{p'}(G)\To G/O_{p'}(G)\To G/F_p(G)\To 1.$$

$G/F_p(G)$ lies in the variety $\var{Ab}_{p-1}$ by Lemma \ref{abp} and $F_p(G)/O_{p'}(G)$ is a $p$-group by Lemma \ref{fpg}.
\end{proof}

\begin{thm}\label{mainthm}
Let $H$ be a subgroup of $F$, then $Cl_{\var{Su}}(H)=\bigcap_{p\in\mathbb{P}}Cl_{\var{H}_p}(H)$.
\end{thm}
\begin{proof}
$Cl_{\var{Su}}(H)\subset Cl_{\var{H}_p}(H) $ since $\var{H}_p$ is a subvariety of $\var{Su}$ for any $p\in \PP$.

If there exists an element $g\in \bigcap_{p\in\mathbb{P}}Cl_{\var{H}_p}(H)-Cl_{\var{Su}}(H)$, then $g$ satisfies both of the following two conditions:

(1) for any $p\in \PP$, $S\in \var{H}_p$  and any homomorphism $\psi: F\to S$, $\psi(g)\in \psi(H)$;

(2) there exist $G\in \var{Su}$ and a homomorphism $\phi: F\to G$ such that $\phi(g)\notin \phi(H)$.

Denote the quotient homomorphism $G\to G/O_{p'}(G)$ by $\tau_p$ for $p\in \pi(G)$. By Lemma \ref{fp},  $G/O_{p'}(G)$ lies in the variety  $\var{H}_p$, thus $\tau_p\circ\phi(g)\in \tau_p\circ\phi(H)$ by the condition (1) above.

We have $\phi(g)\in \tau_p^{-1}(\tau_p(\phi(H)))=\phi(H)O_{p'}(G)$ for each $p\in \pi(G)$, then $$\phi(g)\in \bigcap_{p\in \pi(G)} \phi(H)O_{p'}(G)=\phi(H)$$ by Lemma \ref{fff}, contrary to the condition (2) above.
\end{proof}

\begin{cor}\label{mainres}
Let $H$ be a finitely generated subgroup of $F$, then the $\var{Su}$-closure of $H$ is finitely generated.
\end{cor}
\begin{proof}
Fix a basis $A$ of $F$. By Corollary \ref{hpclo}, each $\var{H}_p$-closure of $H$ lies in the $A$-fringe $\mathcal{O}_A(H)$ of $H$ which is a finite set, then by Theorem \ref{mainthm}, the $\var{Su}$-closure of $H$ is the intersection of finitely many finitely generated subgroups, and so it is also finitely generated.
\end{proof}

\noindent\textbf{Acknowledgements.} The second named author acknowledges partial support from the National Natural Science Foundation of China (Grant No. 12271385).

\end{document}